\documentclass{amsart}[12pt]
\usepackage{amsmath}
\usepackage{amssymb}
\usepackage{amsfonts}
\usepackage{graphicx}
\usepackage{texdraw}
\usepackage{graphpap}
\usepackage{wasysym}
\usepackage{enumitem}
\usepackage[OT2,T1]{fontenc}

\newtheorem{thm}{Theorem}[section]
\newtheorem{mth}{Theorem}

\newtheorem{lem}[thm]{Lemma}

\theoremstyle{definition}

\theoremstyle{remark}
\newtheorem*{defn}{Definition}
\newtheorem*{rem}{Remark}

\numberwithin{equation}{section}

\newcommand{\nbh}{\Omega}
\newcommand{\ti}{\textit{t.i.}}
\newcommand{\lti}{\textit{u.t.i.}}
\newcommand{\bfR}{\mathbf{R}}

\newcommand{\calF}{{\mathcal F}}
\newcommand{\calR}{{\mathcal R}}
\newcommand{\calW}{{\mathcal W}}

\newcommand{\bfK}{{\mathbf K}}
\newcommand{\bulk}{\operatorname{bulk}}
\newcommand{\Int}{\operatorname{Int}}
\newcommand{\La}{L}

\newcommand{\eqpot}{\check{Q}}
\newcommand{\config}{\Theta}

\newcommand{\erfc}{\operatorname{erfc}}

\newcommand{\Bulk}{\operatorname{bulk}}

\newcommand{\R}{{\mathbb R}}

\newcommand{\C}{{\mathbb C}}

\newcommand{\fii}{{\varphi}}
\newcommand{\const}{\mathrm{const.}}
\newcommand{\eps}{{\varepsilon}}

\newcommand{\re}{\operatorname{Re}}
\newcommand{\im}{\operatorname{Im}}

\newcommand{\Pol}{\mathcal{W}}

\newcommand{\Prob}{{\mathbf{P}}}

\newcommand{\Tr}{\operatorname{Tr}}
\renewcommand{\d}{{\partial}}
\newcommand{\dbar}{\bar{\partial}}
\newcommand{\1}{\chi}

\newcommand{\dist}{\operatorname{dist}}
\newcommand{\supp}{\operatorname{supp}}

\newcommand{\Lap}{\Delta}
\newcommand{\lnorm}{\left\|}
\newcommand{\rnorm}{\right\|}
\newcommand{\drop}{S}
\def\norm#1{\lnorm {#1} \rnorm}
\def\lpar{\left (}
\def\rpar{\right )}
\def\labs{\left |}
\def\rabs{\right |}
\def\babs#1{\labs {#1} \rabs}

\begin{document}

\title{A density theorem for weighted Fekete sets}


\author{Yacin Ameur}

\address{Yacin Ameur\\
Department of Mathematics\\
Faculty of Science\\
Lund University\\
P.O. BOX 118\\
221 00 Lund\\
Sweden}

\email{Yacin.Ameur@maths.lth.se}

\begin{abstract} We prove a result concerning the spreading of weighted Fekete points in the plane.
\end{abstract}

\maketitle

Consider a system $\{\zeta_j\}_1^n$ of identical point-charges in the complex plane $\C$, subject to an external field $nQ$, where $Q$ is a suitable real-valued function, which we call the "external potential''. The energy of the system is taken to be
\begin{equation}\label{energy}H_n(\zeta_1,\ldots,\zeta_n)=\sum_{j\ne k}\log\frac 1 {\babs{\,\zeta_j-\zeta_k\,}}
+n\sum_{j=1}^nQ(\zeta_j).\end{equation}
A configuration $\calF_n=\{\zeta_j\}_1^n$ which minimizes $H_n$ is called an $n$-\textit{Fekete set} in external potential $Q$; the points of a Fekete set are called \textit{Fekete points}.

The analog of $H_n$ for continuous charge distributions (measures) $\mu$
is the \textit{weighted logarithmic
energy} in external potential $Q$, defined by
\begin{equation}\label{egy}I_Q[\mu]=\iint_{\C^2}\log \frac 1 {\babs{\,\zeta-\eta\,}}\, d\mu(\zeta)d\mu(\eta)+
\int_\C Q\, d\mu.\end{equation}
The classical \textit{equilibrium measure} $\sigma$ minimizes $I_Q$ amongst all Borel probability measures. This measure is absolutely continuous and has the form
\begin{equation}\label{eqm}d\sigma(\zeta)=\chi_S(\zeta) \Lap Q(\zeta)\, dA(\zeta),
\end{equation}
where the set $S:=\supp\sigma$ is called the \textit{droplet} in external field $Q$.
(See below for further notation.)

 A well-known "discretized'' result asserts that that the normalized counting measures $\mu_n=\frac 1 n\sum_1^n \delta_{\zeta_j}$ at Fekete points converge to the equilibrium measure $\sigma$, as $n\to\infty$. (See e.g. \cite{ST}, Section III.1.)

In this note, we study finer structure of the distribution of Fekete points. We shall prove that Fekete points are maximally spread out in $S$ relative to the conformal metric $ds^2=\Lap Q(\zeta) |d\zeta|^2$ in the
sense of Beurling-Landau densities. Our results generalize those of \cite{AOC}.

\subsection*{Notation} We write $\Lap=\frac 1 4(\d^2/\d x^2+\d^2/\d y^2)$ for $1/4$ times the standard Laplacian and
$dA(z)=d^2\,z/\pi$ is Lebesgue measure in $\C$ normalized so that the unit disk has measure $1$.
The symbols $\d=\frac 1 2(\d/\d x-i\d/\d y)$ and $\dbar=\frac 1 2(\d/\d x+i\d/\d y)$ for the complex derivatives; thus $\Lap=\d\dbar$. The characteristic function of a set $E$ is denoted $\chi_E$. We write $\bar{z}$ or $z^*$ for the complex conjugate of $z$.

\section{Introduction; Main Results}

\subsection{Precise assumptions on the external potential} \label{ptd}
We assume throughout that $Q$ is lower semicontinuous on $\C$ with values in $\R\cup\{+\infty\}$.
Let $\Sigma_0=\{Q<\infty\}$; we assume that the interior $\Int \Sigma_0$ is dense in $\Sigma_0$ and that
$Q$ is real-analytic on $\Int\Sigma_0$. Moreover, $Q$ is assumed to grow rapidly at $\infty$ in the sense that \begin{equation}\label{groww}\liminf_{\zeta\to\infty}\frac {Q(\zeta)}{\log|\,\zeta\,|^{\,2}}>1.\end{equation}

In addition, we make the standing assumption that the droplet $S=\supp\sigma$ is contained in $\Int\Sigma_0$ and that
$Q$ is strictly subharmonic in a neighbourhood of $S$.
This means that Sakai's regularity theorem can be applied; it implies that the boundary $\d S$ has important analytical properties which we recall below. More details about the role of real-analyticity, as well as a formal derivation of the relevant properties of the boundary from Sakai's theory in \cite{Sa,Sa2}, can be found in the concluding remarks (Section \ref{concrem}).

\subsection{Droplets} \label{drp}
Let $p_*\in \d S$ and let $\fii$ be a conformal map from the upper half-plane $\C_+$ to the component $U$ of $S^c=\C\setminus S$ which has $p_*\in \d U$ and $\fii(0)=p_*$. It follows from Sakai's regularity theorem
that the mapping $\fii$ extends analytically across
$\d U$. This leaves three mutually exclusive possibilities: (i) $\fii'(0)\ne 0$ and $\fii(x)\ne p_*$ for all $x\in\R\setminus\{0\}$,
(ii) $\fii'(0)=0$, or (iii) $\fii'(0)\ne 0$ and $\fii(x)=p_*$ for some $x\in\R\setminus\{0\}$.

In case (i), $p_*$ is a \textit{regular} boundary point
meaning that there is a neighbourhood $D$ of $p_*$ such that $D\cap \d S$ is a single real-analytic arc. In cases
(ii) and (iii) we speak of a \textit{singular} boundary point. Case (ii) means
that $p_*$ is a conformal \textit{cusp} pointing into $S^c$ (viz. out of $S$),
and (iii) says that
$p_*$ is a \textit{double point}.

Following \cite{AKMW} we shall now define an integer $\nu\ge 3$ which we call the \textit{type} of a singular boundary point $p_*\in \d S$.
Fix a parameter $T>0$ and let $p_n$ be a point in $S$ of distance $T/\sqrt{n\Lap Q(p_*)}$ from the boundary, which is
closest to $p_*$. If $p_*$ is a cusp, then $p_n$ is unique, while there are two choices if $p_*$ is a double point.

\begin{figure}[ht]
\begin{center}
\includegraphics[width=.40\textwidth]{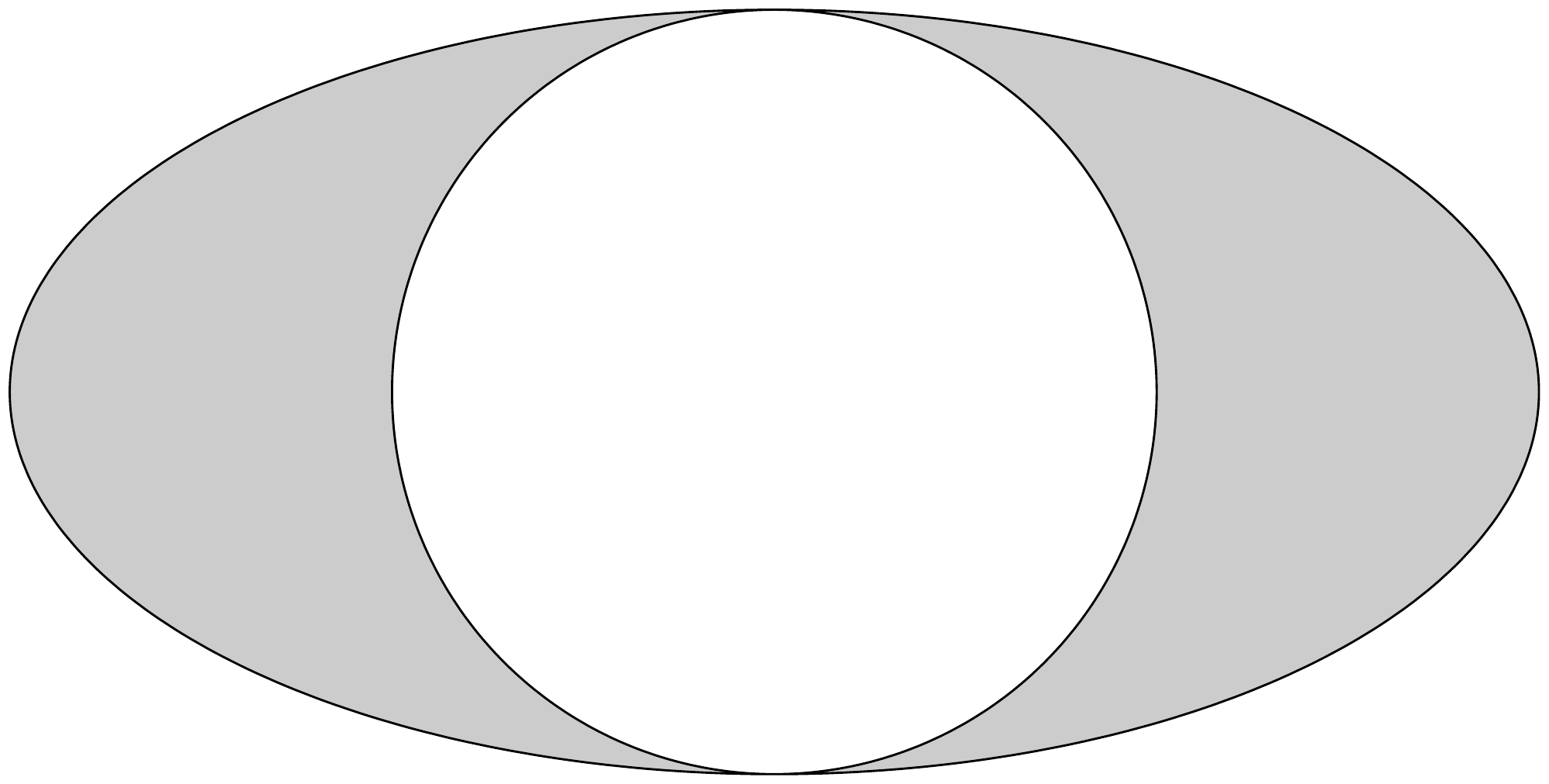}
\hspace{.05\textwidth}
\includegraphics[width=.40\textwidth]{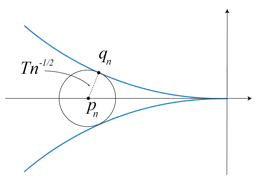}
\end{center}
\caption{A droplet with two double points.
On the right we see a moving point $p_n$ approaching a singular boundary point.}
\label{fig1}
\end{figure}

We claim that there is a unique integer $\nu\ge 3$ such that $\dist(p_n,p_*)\sim n^{-1/\nu}$. Here
"$a_n\sim b_n$'' means that there is a constant $c>0$ such that
$ca_n\le b_n\le c^{-1}a_n$.

To prove existence of such a $\nu$, let us first assume that $\d U$
has a cusp at $p_*=0$. We can assume that a conformal
map $\fii:\C_+\to U$ satisfies
$$\fii'(z)=z+a_2z^2+\cdots
+(a_{\nu-1}+ib)z^{\nu-1}+\cdots$$
where $a_j$ and $b$ are real and $b\ne 0$. Then
$$\fii(z)=\frac 1 2 z^2+\frac {a_2} 3 z^3+\cdots
+\frac {a_{\nu-1}+ib}
{\nu}z^{\nu}+\cdots.$$
Writing $\fii=u+iv$, we see that
$$u(x)=\frac 1 2x^2+\cdots,\quad v(x)=\frac b {\nu}x^{\nu}+\cdots,\quad (x\in\R).$$
This easily gives that $\dist(p_n,0)\sim n^{-1/\nu}$, i.e. the cusp $p_*=0$ has type $\nu$.

 Finally, if $p_*$ is a double point, then there are arcs of $\d S$ which meet at $p_*$ and are tangent to each other there. There are then two closest
points $p_n',p_n''$ to $p_*$ which are at distance $T/\sqrt{n\Lap Q(p_*)}$ from the boundary, and
their distance from $p_*$ is
$\sim n^{-1/(4k)}$ where $k$ is the order of tangency of the two arcs. Thus the type of $p_*$ is $\nu=4k$.

\subsection{Coulomb gas ensembles} Fix a parameter $\beta>0$, the "inverse temperature''.
By the \textit{Boltzmann-Gibbs distribution} on $\C^n$ we mean the measure
\begin{equation}\label{bgd}d\,\Prob_n(\zeta_1,\ldots,\zeta_n)=\frac 1 {Z_n}e^{-\beta H_n(\zeta_1,\ldots,\zeta_n)}\, dA^{\otimes n}(\zeta_1,\ldots,\zeta_n).\end{equation}
Here $H_n$ is the energy \eqref{energy} and $Z_n$ is a normalizing constant, chosen so that $\Prob_n$ is a probability measure. A configuration (or "system'', "point-process'') $\{\zeta_j\}_1^n$ picked randomly with respect to $\Prob_n$ is known as a \textit{Coulomb gas ensemble} in external potential $Q$. By figure of speech, the system
may be conceived as a random perturbation of a Fekete set, with better compliance the lower is the temperature, i.e., the larger that $\beta$ is.

When $\beta=1$ the system is determinantal, i.e.,  we can write
\begin{equation}\label{det}d\,\Prob_n(\zeta_1,\ldots,\zeta_n)=\det\left( \bfK_n(\zeta_j,\zeta_k)\right)_{j,k=1}^n\,
dA^{\otimes n}(\zeta_1,\ldots,\zeta_n),\end{equation}
where $\bfK_n$ is called a \textit{correlation kernel}. (See \cite{ST},
IV.7.2., cf. \cite{M,PS}.)
 The kernel $\bfK_n$
is a Hermitian function, which is determined up to multiplication by cocycles.

Throughout, a continuous function $f:\C^2\to\C$ is called \textit{Hermitian} if $f(z,w)=f(w,z)^*$. A Hermitian function $c$ is termed a \textit{cocycle} if $c(z,w)=g(z)g(w)^*$ for some unimodular function $g$.

\subsection{Rescaling} \label{rescal} Let $\{\zeta_j\}_1^n$ be a configuration of points in $\C$. Further, let
$p=p_n$ be a (possibly moving) point in $S$
and $\theta=\theta_n$ a sequence of real numbers. \footnote{The strings "moving point $p=p_n$'' and
"sequence $p=(p_n)$'' are interchangeable.}
We shall always choose the phase $\theta=\theta_n$ so that $e^{i\theta}$ is one of the normal directions at a boundary point $q_n\in\d S$ which is closest to $p_n$. By convention, we will take $e^{i\theta}$ to be the \textit{outer} normal direction at $q_n$ except in special cases when $p_n$ is close to a singular boundary point.

We rescale about $p=p_n$ as follows
\begin{equation}\label{rmap}z_j=e^{-i\theta}\sqrt{n\Lap Q(p)}(\zeta_j-p),\qquad j=1,\ldots,n\end{equation}
and consider the rescaled system $\{z_j\}_1^n$. If the system $\{\zeta_j\}_1^n$
is picked randomly with respect to $\Prob_n$,
then we consider $\{z_j\}_1^n$ as a point-process in $\C$ whose law is the image of $\Prob_n$ under the scaling \eqref{rmap}. In particular, if $\beta=1$, then $\{z_j\}_1^n$ is a determinantal process with correlation kernel
$$K_n(z,w)=\frac 1 {n\Lap Q(p)}\bfK_n(\zeta,\eta),\quad \text{where}\quad
\begin{cases}z&=e^{-i\theta}\sqrt{n\Lap Q(p)}(\zeta-p)\cr w&=e^{-i\theta}\sqrt{n\Lap Q(p)}(\eta-p)\cr\end{cases}.$$

\subsection{Regimes and translation invariance}
By the \textit{Ginibre kernel}, we mean the Hermitian function
$$G(z,w)=e^{z\bar{w}-|z|^2/2-|w|^2/2}.$$
The following result follows by a normal families argument, see \cite{AKM}.

\begin{lem} \label{cpthm}
There is a sequence $c_n$ of cocycles such that every subsequence of
the sequence $c_nK_n$ has a locally convergent subsequence. Each limit point $K$ takes the form $K=G\Psi$ where
$\Psi(z,w)$ is some Hermitian-entire function.
\end{lem}

Following \cite{AKM}, we will refer to a limit point $K$ in Lemma \ref{cpthm} as a "limiting kernel''.

\begin{defn}
We will consider moving points $p=p_n$ in $S$ of three types:
\begin{enumerate}[label=(\roman*)]
  \item \label{a} If $\liminf\sqrt{n}\dist(p_n,\d S)=\infty$, we say that $p$ is in the
  \textit{bulk} and write $p\in \Bulk S$.
\item \label{b} Suppose that $\limsup \sqrt{n}\dist(p_n,\d S)<\infty$,
and, in the presence of singular boundary points $p_*$, $\liminf n^{1/\nu}\dist(p_n,p_*)=\infty$ where $\nu\ge 3$ is the type of the singular point $p_*$. In this case we say that $p$ belongs of the \textit{regular boundary regime} and write $p\in \d^* S$.
\item \label{c} Suppose that $\limsup n^{1/\nu}\dist(p_n,p_*)<\infty$ where $p_*$ is a singular boundary point of type $\nu$. We then say that $p$ belongs to the \textit{singular boundary regime} and write $p\in\d'S$.
\end{enumerate}
\end{defn}

The types \ref{a} and \ref{b} were introduced in \cite{AOC}.
For type \ref{a}, one has the following result, which is a generalized version of
 the "Ginibre($\infty$)-limit''. See \cite{AOC} or \cite{AKM}, Section 7.6 for proofs.

\begin{lem} \label{gininfty}
If $p$ belongs to the bulk, then
$K=G$ for each limiting kernel.
\end{lem}

In general, a limiting kernel $K=G\Psi$ is called \textit{translation invariant} (in short: \ti)
if $\Psi(z,w)=\Phi(z+\bar{w})$ for some entire function $\Phi$. It is natural to conjecture that
any limiting kernel is \ti \footnote{Since our standing assumptions preclude points in $S$ where $\Lap Q=0$; see comments below.}
 It was observed in \cite{AKM} that this is the case when the potential is radially symmetric, and the computations in the paper
\cite{LR} show that this is true also for the "ellipse ensemble'', i.e. for the potentials $Q_t(\zeta):=|\,\zeta\,|^{\,2}-t\re(\zeta^2)$, $0<t<1$.
For other potentials the conjecture is still open. See the papers \cite{AKM,AKMW} for several related comments.

\begin{defn}
We say that the potential $Q$
is \textit{universally translation invariant} (in short: \lti) if all limiting kernels
are \ti
\end{defn}

Thus all radially symmetric potentials, as well as the ellipse potentials, are \lti

\smallskip

Following \cite{AKM}, we say that a limiting
kernel $K$ satisfies the
\textit{mass-one equation} if
\begin{equation}\label{moeq}\int_\C\babs{\,K(z,w)\,}^{\,2}\, dA(w)=R(z),\qquad
(z\in\C).\end{equation}
Here $R(z)=K(z,z)$ is the \textit{limiting one-point function} pertaining to the limiting kernel $K$. A limiting one-point function which does not vanish identically is called \textit{nontrivial}.

The following \textit{zero-one law}, proved in \cite{AKM}, is fundamental for what follows.

\begin{lem} \label{muck} A nontrivial limiting one-point function is everywhere strictly positive. Moreover, each \ti~ limiting kernel satisfies the mass-one equation.
\end{lem}

\subsection{Beurling-Landau densities}
Consider a family $\config=\{\config_n\}_{n=1}^\infty$ where $\config_n=\{\zeta_{nj}\}_1^n$ is an $n$-point configuration in $S$.

Write $D(a,r)$ for the Euclidean disk of center
$a$ and radius $r$.
Given a moving point $p=p_n\in S$ and a positive parameter $\La$, we denote
\begin{align}\label{anz1}A_n(p,\La)&=D\left(p_n,\frac \La{\sqrt{n\Lap Q(p_n)}}\right),\quad &N_{n}(p,\La;\config)&=\#\left\{\config_n\cap A_n(p,\La)\right\}.
\end{align}
We define the \textit{upper Beurling-Landau density} of $\config$ at a moving point $p=p_n$ by
$$D^+(\config,p)=\limsup_{\La\to\infty}\limsup_{n\to\infty}
\frac {N_n(p,\La;\config)}{\La^{\,2}}.$$
The corresponding lower density
$D^-(\config,p)$ is defined by replacing the two "$\limsup$'' by "$\liminf$''. If $D^+(\config,p)=D^-(\config,p)$ we write
$D(\config,p)$ for the common value, and speak of the Beurling-Landau density of $\config$ at $p$.

It will be convenient to also use another measure of spreading of a family $\config$.
For a point $\zeta_{nj}\in\config_n$ we will denote the distance to its closest neighbor by
$$d_n(\zeta_{nj}):=\min_{k\ne j}\babs{\,\zeta_{nk}-\zeta_{nj}\,}.$$
We shall consider the quantity
$$\Delta(\config):=\liminf_{n\to\infty}\min\left\{
\sqrt{n\Lap Q(\zeta_{nj})}\,d_n(\zeta_{nj});\, j=1,\ldots,n\right\}.$$
We refer to $\Delta(\config)$ as the \textit{asymptotic separation constant} of the family
$\config$.

Now consider a family $\calF=\{\calF_n\}$ of Fekete sets. It is well known that
$\calF\subset S$ (see \cite{ST}, Theorem III.1.2).

The union of the following results comprise the "density theorem'' of this note.

\begin{mth} \label{THAA} $\Delta(\calF)\ge 1/\sqrt{e}$.
\end{mth}

\begin{mth}\label{THAAA} If $p\in \bulk S$, then $D(\calF,p)=1$.
\end{mth}

\begin{mth} \label{THA}
If $p\in \d^*S$ and $Q$ has the \lti-property, then $D(\calF,p)=1/2$.
\end{mth}

\begin{mth} \label{THB} If $p\in\d' S$, then $D(\calF,p)=0$.
\end{mth}

\subsubsection*{Comments} Theorem \ref{THAA} with some unspecified separation constant was proved independently in the papers \cite{AOC} and \cite{NS}; our proof here is quite similar to the one in \cite{AOC}.
The problem of finding the exact value of $\Delta(\calF)$ comes close to the (Abrikosov's) conjecture that Fekete points should be organized as nodes of a honeycomb lattice, relative to the conformal metric (perhaps with some slight disturbances near the boundary). The bound $1/\sqrt{e}=0.606...$ may be regarded as a first crude step in this direction. In view of the structure of the honeycomb lattice, it could be speculated that the exact value of $\Delta(\calF)$ might be $\sqrt{2/\sqrt{3}}=1.074..$.

 A similar (much earlier) separation result, for Fekete points on a sphere, was obtained by Dahlberg \cite{D}, and the corresponding separation constant has since then been the subject of various investigations. To our knowledge the current record is found in \cite{Dr}. Cf. also \cite{KSS,MP}.

Theorem \ref{THAAA} was in effect proved in \cite{AOC} (Section 7) but we give a simpler proof here, depending on results from \cite{AKM}.

The special case of Theorem \ref{THA} for the Ginibre potential $Q=\babs{\,\zeta\,}^{\,2}$
was shown in \cite{AOC}, where it was also conjectured that the result be true in general.

 A classical result asserts weak convergence, in the sense of measures, of the normalized counting measures at Fekete points to the equilibrium measure, i.e., the measures
$\mu_n:=n^{-1}\sum_1^n\delta_{\zeta_{nj}}$ converge as measures to $\sigma$ as $n\to\infty$. Cf. \cite{ST}, Section III.1. The convergence of the counting
measures has been generalized in various directions, see for instance \cite{AOC,BBNY,BBW,Dr,HSa,KSS,L,Le,LOC,MP,NS,PeS,POC,SS1} and the references there.
One of those directions concerns the
case of a complex line bundle over a compact manifold \cite{BBW,LOC,POC}.
Compactness is manifest for many important questions, e.g. for the distribution of Fekete points on spheres. The "droplet'' is then the entire manifold, i.e., there is only bulk. In this case, the rate of weak convergence $\mu_n\to\sigma$ has been quantified in terms of the Wasserstein metric, see \cite{LOC}.
However, the compact case is different from ours, and our present density results are of another type. Our setting is more directly related to the situation in \cite{ST} and e.g. \cite{SS1,NS,PeS}.

In the presence of \textit{bulk singularities},  i.e. isolated points $p^*\in\Int S$ where $\Lap Q(p_*)=0$, \footnote{There may also be boundary points $p_*$ at which
$\Lap Q(p_*)=0$; see \cite{BGM} for an example. It seems that this kind of "doubly singular boundary points'' have not yet been completely classified.} one can introduce a fourth "bulk singular regime'', where the distribution of Fekete points will be "sparse'',  depending on the type of the singularity. In such a regime, translation invariance is lost, and Ward's equation takes a different form, see \cite{AS}, cf. \cite{AKM}, Section 7.3. It is also possible to introduce certain types of logarithmic singularities in the potential, as in
\cite{AK}, giving different sorts of effects for the Fekete points. We will return to this in a later publication.

From a technical point of view, we
continue to develop methods from
the paper \cite{AOC}. We rely on techniques of sampling and interpolation in spaces of weighted polynomials, which in turn builds on techniques developed in the papers \cite{MMOC,MOC,POC} (cf. also the book \cite{Se}). We shall also use elements of the method of rescaled Ward identities from the paper \cite{AKM}, and the technique of moving points from \cite{AKM,AKMW,AOC}.

Finally, as in \cite{AOC}, it is worth to point out an interesting feature of our method; it uses properties of a  temperature $1/\beta=1$ Coulomb gas, to obtain information about the Beurling-Landau density at temperature $1/\beta=0$.

\subsection{Notational conventions}  Whenever an unspecified measure
is indicated, such as in "$\int f$'', or "$L^p$'', the measure is understood to be $dA$. An unspecified norm $\|\cdot\|$ always denotes
the norm in $L^2$.

Following \cite{AKM}, we shall
use boldface letters $\bfR$, $\bfK$, etc. to denote objects pertaining to the non-rescaled point-process $\{\zeta_j\}_1^n$ from the $\beta=1$ ensemble
associated with potential $Q$. Corresponding objects pertaining to the rescaled system $\{z_j\}_1^n$ given by \eqref{rmap} will be denoted by plain symbols, $R$, $K$, etc.
In particular, we denote the \textit{$k$-point function} of the system $\{\zeta_j\}_1^n$ by
$$\bfR_{n,k}(\zeta_1,\ldots,\zeta_k)=\det\left(\bfK_n(\zeta_i,\zeta_j)\right)_{i,j=1}^k$$
and the rescaled version will be written
$$R_{n,k}(z_1,\ldots,z_k)=\frac 1{(n\Lap Q(p))^k}\bfR_{n,k}(\zeta_1,\ldots,\zeta_k).$$
When $k=1$ we omit the subscript, and simply write
$$\bfR_n(\zeta)=\bfR_{n,1}(\zeta)=\bfK_n(\zeta,\zeta)\quad ,\quad
R_n(z)=R_{n,1}(z)=K_n(z,z).$$
Finally, it is convenient to denote by
$$B_n(z,w)=\frac {\babs{\,K_n(z,w)\,}^{\,2}}{R_n(z)}$$
the (rescaled) \textit{Berezin kernel} rooted at $z$. Note that
$\int_\C B_n(z,w)\, dA(w)\equiv 1$.

\section{Some preparations}

In this section, we provide various apriori estimates for weighted polynomials and limiting kernels.

\subsection{Weighted polynomials} By a \textit{weighted polynomial} of degree $n$, we mean a function of the form $f=p\cdot e^{-nQ/2}$ where $p$ is a (holomorphic) polynomial
of degree at most $n-1$. We write $\Pol_n$ for the space of weighted polynomials regarded as a subspace of $L^2$.

It is
well-known that the reproducing kernel $\bfK_n(\zeta,\eta)$ for $\Pol_n$ is a correlation kernel for the random normal matrix ensemble associated with the potential $Q$, i.e., with this choice of $\bfK_n$, we have the identity \eqref{det} for $\beta=1$. See e.g. \cite{ST}, Section IV.7.2.

\subsection{The Bernstein inequality}
The following lemma is a sharper version of a result from \cite{AOC}.

\begin{lem} \label{bern} Suppose that $f\in\Pol_n$ and $\zeta\in S$.
If $f(\zeta)\ne 0$, then
\begin{equation}\babs{\,\nabla \babs{\,f(\zeta)\,}\,}\le \sqrt{en\Lap Q(\zeta)}\,\left(1+O(1/\sqrt{n})\right)\left\|\,f\,\right\|_{L^\infty}.\end{equation}
Here the $O$-constant is uniform for $\zeta\in S$.
\end{lem}

\begin{proof}
Write $f=pe^{-nQ/2}$ where $p$ is a polynomial of degree at most $n-1$. Also fix $\zeta\in S$
with $\zeta\not\in\calF_n$ (so $\babs{\,f\,}$ is differentiable near $\zeta$). We write
$$H_\zeta(\eta)=Q(\zeta)+2\d Q(\zeta)\cdot(\eta-\zeta)+\d^2 Q(\zeta)\cdot (\eta-\zeta)^2$$
and
$$h_\zeta(\eta)=\re H_\zeta(\eta).$$
By Taylor's formula,
\begin{equation}\label{nb}\begin{split}n\babs{\,Q(\eta)-h_\zeta(\eta)\,}&\le n\Lap Q(\zeta)\babs{\zeta-\eta}^2+O\left(1/\sqrt{n}\right)\\
&=1+O\left(1/\sqrt{n}\right)\quad \text{when}\quad
\babs{\,\eta-\zeta\,}=1/\sqrt{n\Lap Q(\zeta)}.\\
\end{split}\end{equation}

Next observe that
\begin{equation}\label{1o}\babs{\nabla\left(\babs{\,p\,}e^{-nQ/2}\right)(\eta)}=
\babs{\,p'(\eta)-n\cdot\d Q(\eta)\cdot p(\eta)\,}e^{-nQ(\eta)/2},
\end{equation}
and
\begin{equation}\label{2o}\begin{split}
\babs{\nabla\left(\babs{p}e^{-nh_\zeta/2}\right)(\eta)}&=
\babs{\,p'(\eta)-n\cdot\d h_\zeta(\eta)\cdot p(\eta)\,}e^{-nh_\zeta(\eta)/2}\\
&=\babs{\frac d {d\eta}\left(pe^{-nH_\zeta/2}\right)(\eta)}.\\
\end{split}\end{equation}
The expressions in \eqref{1o} and \eqref{2o} are identical when $\eta=\zeta$.

Let $\gamma$ be the circle centered at $\zeta$ with radius $1/\sqrt{n\Lap Q(\zeta)}$. By a Cauchy estimate,
\begin{equation}\label{3o}\begin{split}
\babs{\frac d {d\eta}\left(pe^{-nH_\zeta/2}\right)(\zeta)}&=
\frac 1 {2\pi }\babs{\int_\gamma\frac {p(\eta)e^{-nH_\zeta(\eta)/2}}{(\zeta-\eta)^2}\, d\eta}\\
&\le\frac {n\Lap Q(\zeta)} {2\pi}\int_\gamma\babs{p(\eta)}e^{-nh_\zeta(\eta)/2}\,\babs{d\eta}.\\
\end{split}
\end{equation}
In view of \eqref{nb}, the far right side is dominated by
\begin{align*}\frac {n\Lap Q(\zeta)} {2\pi}e^{1/2+O(1/\sqrt{n})}\int_\gamma\babs{p(\eta)}e^{-nQ(\eta)/2}\,\babs{d\eta}&\le \sqrt{n\Lap Q(\zeta)}e^{1/2+O(1/\sqrt{n})}\sup_{\eta\in\gamma}\babs{f(\eta)}\\
&\le \sqrt{n\Lap Q(\zeta)}\sqrt{e}\left[1+o(1)\right]\left\|\,f\,\right\|_{L^\infty}.\end{align*}
The proof is complete.
\end{proof}

\subsection{Auxiliary estimates} We will use the following standard facts.

\begin{lem} \label{damp}
If $f\in\Pol_n$ and $\babs{\,f\,}\le 1$ on $S$, then $\babs{\,f\,}\le 1$ on $\C$.
\end{lem}

\begin{lem} \label{seppo} If $f=ue^{-nQ/2}$ where $u$ is holomorphic and bounded in $D(\zeta,c/\sqrt{n})$ and $\Lap Q\le K$ in $D(\zeta,c/\sqrt{n})$, then there is a constant $C=C(K,c)$ such that
$$\babs{\,f(\zeta)\,}^{\,2}\le Cn\int_{D(\zeta,c/\sqrt{n})}\babs{\,f\,}^{\,2}\, dA.$$
\end{lem}

\begin{lem} \label{penttinen} Let $p=p_n$ be a moving point and $R_n$ the corresponding rescaled $1$-point function. Then there is a constant $C$ such that $R_n\le C$ on $\C$.
\end{lem}

We refer to \cite{AKM}, Section 3 for a discussion of proofs.

\subsection{Lower bounds for the one-point function} For a point $\zeta\in S$ we
denote the distance to the boundary by
$$\delta(\zeta)=\dist(\zeta,\d S).$$
It will be convenient to introduce a rescaled version of the distance to the boundary,
of a moving point $p=p_n$, by
$$a_n(p):=\sqrt{n\Lap Q(p_n)}\delta(p_n).$$

The following lemma gives apriori bounds for a limiting 1-point function.

\begin{lem} \label{leftlem2} Let $p=p_n$ be a moving point in $\bulk S$ or in $\d^* S$.
Suppose that, along a subsequence $n_k$, the limit $a=\lim_{k\to\infty} a_{n_k}(p)$
exists, being possibly $+\infty$. Let $R=\lim R_{n_{k_l}}$ be a limiting one-point function pertaining to the subsequence $n_k$. There are then positive constants
$C$ and $c$ such that
\begin{equation}\label{vuko}\babs{\,R(z)-\chi_{(-\infty,a)}(x)\,}\le Ce^{-c(x-a)^2},\qquad (x=\re z).\end{equation}
\end{lem}

(We here we use the convention $e^{-\infty}=0$.)

\begin{rem} For the proof of the lemma, we could simply refer to Theorem C in \cite{AKM} and Theorem 3.1 in \cite{AKMW}, but since we shall also use a slightly more precise estimates for finite $n$ one-point functions $R_n$, we give more details.
\end{rem}

\begin{proof}
 Our starting point is Theorem 5.4 in \cite{AKM}, which says that
\begin{equation}\label{akmest}\babs{\,\bfR_n(p_n)-n\Lap Q(p_n)\,}\le C\left(1+n\exp\left\{-c\cdot a_n(p)^2\right\}\right),\qquad p\in S.\end{equation}
Rename the subsequence in the hypothesis from "$n_k$'' to "$n$''
and
consider the rescaled $1$-point function
$$R_n(z)=\frac 1 {n\Lap Q(p_n)}\bfR_n(\zeta),\quad z=e^{-i\theta_n}\sqrt{n\Lap Q(p_n)}(\zeta-p_n).$$
Then
\begin{equation}\label{bjud}a_n(\zeta)=a_n(p)-x+o(1),\quad \zeta\in S,\end{equation}
where $x=\re z$. If $p$ is in the bulk, then $a_n(p)\to\infty$, and it
follows from \eqref{akmest} that $\babs{R_n(z)-1}\le C/n$ with a constant $C$ which can be chosen uniformly
for $z$ in a given compact set. If $p$ is in the regular boundary regime, then \eqref{akmest} and \eqref{bjud}
show that there are numbers $M_n$ with $M_n\to\infty$ such that
\begin{equation}\label{akk2}\babs{\,R_n(z)-1\,}\le C\left(n^{-1}+e^{-c(x-a_n(p))^2}\right),\quad x=\re z < a,\,|\,z\,|\le M_n.\end{equation}
This gives \eqref{vuko} when $x\le a=\lim a_n(p)$. On the other hand, when $x\ge a$, Theorem 3.1 in \cite{AKMW}
can be applied; it implies that any limiting $1$-point function $R$ at $p$ satisfies the estimate \eqref{vuko} for $x\ge a$ where we may take $c=2$.
\end{proof}

We shall say that a family of functions $R:\C\to [0,+\infty]$ is \textit{locally uniformly bounded below} (or \textit{l.u.b.b.}) if for each
$\La>0$ there exists
$\delta=\delta(\La)>0$ such that for all $R$ in the family we have $R(z)> \delta$
when $\babs{\,z\,}< \La$.

\smallskip

For a given $T>0$ we will denote by $X_T$ the set of moving points $p=p_n\in S$ such that $\dist(p_n,\d S)\ge T/\sqrt{n\Lap Q(p_n)}$. Let $\calR_T$ denote the family of all limiting one-point functions which arise at points of $X_T$.

\begin{lem} \label{below1.5}
If $T$ is large enough, then the family $\calR_T$ defined above is l.u.b.b.
\end{lem}

\begin{proof} Take $\delta$ in the interval $0<\delta<1$. For a point $p\in X_T$ the estimate \eqref{akk2} gives
$\babs{\, R_n(0)-1\,}\le C(n^{-1}+e^{-cT^2})\le 1-\delta$ if $n$ and $T$ are large enough. Hence $R(0)\ge \delta$ for all $R\in \calR_T$, and it follows that
$R>0$ everywhere on $\C$, by the zero-one law in Lemma \ref{muck}. If the family of all one-point functions constructed in this way is not l.u.b.b., we can find a point $z_0\in \C$ and a sequence $R^1,R^2,\ldots$ in $\calR_T$ such that $R^n(z_0)\to 0$ as $n\to\infty$. Write $R^j=\lim_{n\to\infty}R^j_n$.

By Lemma \ref{cpthm}, the diagonal sequence $R_n^n$ contains a locally uniformly convergent subsequence with limit $R$ such that $R(z_0)=0$. By the zero-one law we have $R(0)=0$. This is a contradiction, since
$R_n^n(0)\ge \delta$ for all large $n$.
\end{proof}

\begin{rem}
Note that if $p,q$ are moving points with $|\,p_n-q_n\,|\le T/\sqrt{n}$, then by our choice of rescaling, any limiting one point functions at $p$ and $q$ respectively, will be translates of each other, say $R_q(z)=R_p(z+c)$ where $|c|\le \const \cdot T$. This is used to prove the following two lemmas.
\end{rem}

For given $C>0$ and $s\ge 0$ we let $\tilde{X}_{C,s}$ denote the set of moving points $p=p_n$ in $S_s=S+D(0,s/\sqrt{n})$
such that $\dist(p_n,p_*)\ge Cn^{-1/\nu}$ whenever $p_*\in \d S$ is a singular boundary point of type $\nu$.
Let $\tilde{\calR}_{C,s}$ denote the family of one-point functions at points of $\tilde{X}_{C,s}$.

\begin{lem} \label{below}
If $C$ is large enough, then $\tilde{\calR}_{C,s}$ is l.u.b.b.
\end{lem}

\begin{proof} Fix $s\ge 0$. A geometric consideration shows that if $p\in \tilde{X}_{C,s}$ then there is a constant $T$ with $T\to\infty$ as $C\to\infty$ and points $q_n\in S$ with distance at least $T/\sqrt{n}$ to the boundary such that $\babs{\,p_n-q_n\,}\le (T+s)/\sqrt{n}$. The statement now follows from Lemma \ref{below1.5} and the remark above.
\end{proof}

We will need the following, slightly more precise version of Lemma \ref{below}.

\begin{lem}\label{beloww}
Suppose that $p=p_n$ is a moving point in $\tilde{X}_{C,s}$ where $C$ is large enough. Then there are constants
$c>0$ and $n_0$ independent of $p$ such that $\bfR_n(p_n)\ge cn$ when $n\ge n_0$.
\end{lem}

\begin{proof} Fix $0<\delta<1$. Pick $p\in \tilde{X}_{C,s}$ and let $R_n$ be the one-point function rescaled about $p_n$.
By the proof of Lemma \ref{below1.5} there are
$T$ and $n_0$ such that
if $p_n\in S$ has distance at least $T/\sqrt{n}$ to $\d S$, then $R_n(0)\ge \delta$, i.e., $\bfR_n(p_n)\ge \delta n$,
when $n\ge n_0$.

Now put $\delta_n=\frac 1 n \inf_{q\in\tilde{X}_{C,s}}\left\{\bfR_n(q_n)\right\}$ and assume that $\delta_n\to 0$ as $n\to\infty$ along some subsequence.

Fix $s\ge 0$ and a large number $T$. If $C$ is large enough, we can pick $q\in\tilde{X}_{C,s}$ with $\bfR_n(q_n)/n\le 2\delta_n$ and then $p_n\in S$ with
$\dist(p_n,\d S)\ge T/\sqrt{n}$ and $|p_n-q_n|\le (T+s)/\sqrt{n}$. Passing to a subsequence, we can assume that the image of $q_n$ under the rescaling converges to a point $z_0\in \C$ with $|z_0|\le \const(T+s)$,  so
$R(z_0)=0$ for the corresponding limiting $1$-point function about $p$. By the zero-one law, $R\equiv 0$, which gives a contradiction. We conclude that $\delta_n\ge c>0$ for some constant $c>0$ independent of $n$.
\end{proof}

\subsection{Ward's equation and translation invariance}
Let $R$ be a limiting one-point function. The Hermitian function $\Psi$ in the corresponding limiting kernel $K=G\Psi$
is uniquely determined by $R$ by polarization, since $\Psi(z,z)=R(z)$.
If $R$ is nontrivial, then $R$ is everywhere strictly positive by the zero-one law (Lemma \ref{muck}). Hence
the limiting \textit{Berezin kernel}
$$B(z,w):=\frac {\babs{\,K(z,w)\,}^{\,2}}{R(z)}$$ is well-defined and is completely determined by $R$.
We recall from \cite{AKM} that $R$ gives rise to a solution to \textit{Ward's equation}
\begin{equation}\label{ward}\dbar C=R-1-\Lap\log R,\quad \text{where}\quad C(z)=\int \frac {B(z,w)}{z-w}\, dA(w).
\end{equation}

It is convenient to consider
"generalized \ti~kernels'' of the form
\begin{equation}\label{genti}K(z,w)=G(z,w)\Phi(z+\bar{w})\end{equation}
where $\Phi$ is an entire function. We will write $\gamma(z)=(2\pi)^{-1/2}e^{-z^2/2}$ and write
$$\gamma*g(z):=\int_\R\gamma(z-t)g(t)\, dt,\qquad z\in\C,$$
for the "convolution'' of $\gamma$ with a function $g\in L^\infty(\R)$.
The following result from \cite{AKM} summarizes the properties of such kernels which will be needed in the sequel.

\begin{lem} \label{back}
Let $K$ be a kernel of the form \eqref{genti}
where $\Phi$ is an entire function. Then $K$ satisfies Ward's equation if and only if there is an interval
$I\subset \R$ such that $\Phi=\gamma*\1_I$.
Further, $K$ satisfies the mass-one equation if and only if $\Phi=\gamma*\1_e$ for some Borel set $e\subset\R$ of positive measure.
\end{lem}

To apply the lemma, it is convenient to assume that the \lti-property holds, i.e., each limiting kernel $K$ at a moving point $p=p_n$ is translation invariant.

Assuming that $p\in\Bulk S\cup\d^* S$ it follows from the estimates
Lemma \ref{leftlem2} and the representation $\Phi=\gamma*\1_I$ in Lemma \ref{back} that the corresponding function
$\Phi$ is of the form
$$\Phi_m(z)=\gamma*\1_{(-\infty,m)}(z)=F(z-m)$$
for some $m\in \R\cup\{+\infty\}$, where
\begin{equation}\label{freb}F(z):=\int_{-\infty}^0 \gamma(z-t)\, dt=\frac 1 2 \erfc\left(\frac z{\sqrt{2}}\right).\end{equation}
This function was termed the "free boundary plasma function'' in \cite{AKM}. 
(If $m=+\infty$, we interpret $\Phi_m\equiv 1$.)

The limiting kernel pertaining to $\Phi_m$ will be denoted
\begin{equation}\label{ka}K^{m}(z,w):=G(z,w)F(z+\bar{w}-2m).\end{equation}

\begin{rem} For a true limiting kernel we have $m=0$ at almost every fixed regular boundary point, by Theorem D from \cite{AKM}. In general, if $p_n\in S$ is in the regular boundary regime, then $m\ge 0$. This follows from the exterior estimate in \cite{AKMW}. In the sequel, we shall however merely need the fact that $m>-\infty$.
\end{rem}

Our next result shows that the generalized Berezin kernels
$$B^{m}(z,w):=\frac {\babs{K^{m}(z,w)}^2}{K^{m}(z,z)}$$
satisfy the mass-one equation uniformly in $m$.

\begin{lem} \label{uniint} Put
$$\mu(\La)=\mu(\La,z,m):=\int_{|\,w-z\,|\le \La}B^{m}(z,w)\, dA(w).$$
Then $\mu(\La)\to 1$ as $\La\to \infty$ uniformly in $(z,m)$ when $2\re (z-m)\le M$ for some fixed $M<+\infty$.
\end{lem}

\begin{proof} By the mass-one equation, we have
\begin{equation}\label{1ml}1-\mu(\La)=\int_{|\,w-z\,|> \La}\frac {\babs{\,K^{m}(z,w)\,}^{\,2}}{K^{m}(z,z)}\, dA(w).\end{equation}
If $2\re (z-m)\le M<+\infty$ we have (with $F$ the plasma function \eqref{freb})
$$K^{m}(z,z)=F(z+\bar{z}-2m)\ge F(M)>0,$$
so with $C_1=1/F(M)$ we have
\begin{equation}\label{bul}1-\mu(\La)\le C_1\int_{|\,w-z\,|>\La}\babs{\,K^{m}(z,w)\,}^{\,2}\, dA(w).\end{equation}
We now invoke the estimate in \cite{AOC}, Lemma 8.5,
$$\babs{K^{m}(z,w)}\le e^{-|\,z-w\,|^2/2}+e^{-[\re(z-w)]^2/2}H(\im(z-w)),$$
where $H$ is Dawson's function,
$$H(t)=(2\pi)^{-1/2}e^{-t^2/2}\int_0^t e^{x^2/2}\, dx.$$ By standard
asymptotics, $\babs{H(t)}\le C_2(1+|t|)^{-1}$ for all $t\in \R$. (Cf. the proof of Lemma 8.5 in \cite{AOC} or \cite{SO}.) This shows that
\begin{align*}\int_{|w-z|>\La}&\babs{\,K^{m}(z,w)\,}^{\,2}\, dA(w)\\
&\le
\int_{|w-z|>\La}\left[C_3e^{-|z-w|^2/2}+C_4 \frac {e^{-[\re(z-w)]^2}}
{1+[\im(z-w)]^2}\right]\, dA(w).
\end{align*}
The right hand side clearly tends to $0$ as $\La\to\infty$, independently of $(z,m)$. This finishes the proof, in view of \eqref{1ml} and \eqref{bul}.
\end{proof}

\section{Interpolating families and $M$-families}
Consider a sequence $\config=\{\config_n\}$ of $n$-point configurations in $S$,
$$\config_n=\left\{\zeta_{n1},\ldots,\zeta_{nn}\right\}.$$
We recall the definitions of two classes of families from \cite{AOC}.

Fix a real parameter $\rho$ (close to 1) and consider the space $\Pol_{n\rho}$ of weighted polynomials of degree $n\rho$. We can assume that $n\rho$ is an integer.

We say that $\config$ is \textit{$\rho$-interpolating} if there is a constant $C=C(\rho)$ such that for all families of values $c=\{c_n\}_1^\infty$,
$c_n=\{c_{nj}\}_{j=1}^n$
there exists a sequence $f_n\in\Pol_{n\rho}$ such that
$f_n\left(\zeta_{nj}\right)=c_{nj}$, $j=1,\ldots,n$, and
\begin{equation}\label{rip}\left\|\,f_n\,\right\|^{\,2}\le C\frac 1 {n}\sum_{j=1}^n\babs{\,c_{nj}\,}^{\,2}.\end{equation}

The family $\config$ is called \textit{uniformly separated} if there is a constant $s>0$ independent of $n$ such that $$\babs{\,\zeta_{nj}-\zeta_{nj'}\,}\ge \frac s{\sqrt{n}},\quad \text{whenever}\quad j\ne j'.´$$
A number $s$ with this property is a \textit{separation constant} for $\config$.

For a subset $\Omega\subset S$ and a fixed $s>0$ we write
$$\Omega_s:=\Omega+D(0;s/\sqrt{n}).$$
We shall say that $\config$ is of \textit{class $M_\rho$} if there is an $s>0$ such that $\config$ is $2s$-separated and there is a constant $C=C(s,\rho)$ such that
$$\int_{\drop_s}\babs{\,f\,}^{\,2}\le C\frac 1 {n}\sum_{j=1}^n\babs{\,f\left(\zeta_{nj}\right)\,}^{\,2},\qquad f\in\Pol_{n\rho}.$$

Intuitively, interpolating families are sparse and $M$-families are dense.
We shall use two simple facts about these classes.

\begin{lem}\label{seppo2} If $\config$ is $2s$-separated and $\Omega\subset \drop$, then for all $f\in\Pol_n$
$$\frac 1 n\sum_{\zeta_{nj}\in \Omega}\babs{\,f(\zeta_{nj})\,}^{\,2}\le Cs^{-2}\int_{\Omega_s}\babs{\,f(\zeta)\,}^{\,2}\, dA(\zeta),$$
where $C$ depends only on the upper bound of $\Lap Q$ on $\drop_s$.
\end{lem}

\begin{proof}
This is immediate from Lemma \ref{seppo}.
\end{proof}

\begin{thm} \label{seppo3} If $\config$ is a $\rho$-interpolating family contained in $S$, then $\config$ is uniformly separated.
\end{thm}

\begin{proof} This follows from Bernstein's inequality (Lemma \ref{bern}). Details are found in \cite{AOC}, \textsection 3.2.
\end{proof}

\section{Sampling and interpolation families}\label{conop}

In this section, we review the method from \cite{AOC} for estimating upper and lower Beurling-Landau density for a family $\config$ at a moving point $p=p_n$, and extend its scope. The main insight is that essentially the same method used to treat the Ginibre potential in \cite{AOC} can be made to work for any \lti\, potential. To accomplish this, we must of course replace all estimates which depend on the explicit form
of the Ginibre potential; we mention that the mass-one equation plays a key role.
 A related simplification is that we here completely avoid the use of off-diagonal estimates for the correlation kernel.
 To this end, we use an observation found in \cite{LOC}, that certain off-diagonal estimates
 in Section 6 of \cite{AOC} can be replaced by lower bounds for the one-point function.

\medskip

We now introduce the setup. Observe that
the orthogonal projection $P_{n\rho}$ of $L^2$ onto $\Pol_{n\rho}$ can be written
$$P_{n\rho}\left[f\right](\zeta)=\left\langle f,\mathbf{K}_{\zeta}\right\rangle$$
(inner product in $L^2$) where $\bfK=\mathbf{K}_{n\rho}$ is the reproducing kernel for $\Pol_{n\rho}$.

Fix a number $\La\ge 1$.
The \textit{concentration operator} associated with a moving point $p=p_n\in \drop$ is defined by
\begin{equation}\label{conc}T_{n,\La}:\Pol_{n\rho}\to\Pol_{n\rho}\quad :\quad f\mapsto P_{n\rho}\left[\1_{A_n(p,\La)}\cdot f\right].\end{equation}
Here $A_n(p,\La)$ is the neighbourhood of $p_n$ defined in eq. \eqref{anz1}.

\subsection{Trace estimates} In the following, "$\Tr T$'' denotes the trace of an operator $T$ on finite dimensional space.

\begin{lem}\label{cancella} Let
$p=p_n$ be any moving point in $S$. If either $p$ is in $\Bulk S$ or
if $Q$ has the \lti-property, then
\begin{equation}\label{traces}\lim_{\La\to\infty}\lim_{n\to\infty}\frac 1 {\La^{\,2}}\Tr\left(T_{n,\La}-T_{n,\La}^{\,2}\right)= 0.\end{equation}
\end{lem}

\begin{proof} Let $K_{n_k}$ be a subsequence of the kernels $K_n$ which converge locally uniformly to a limit $K$. We do not exclude that $K$ vanishes identically.
In view of Lemma \ref{muck}, we have the "mass-one equation''
$$\int_\C\babs{\,K(z,w)\,}^{\,2}\, dA(w)=R(z),$$
where $R(z)=K(z,z)$ is either everywhere positive or else vanishes identically.

Write $\1=\1_{A_n(p,\La)}$. A change of variables leads to
$$\Tr T_{n,\La}=\int_\C \1\cdot \bfR_{n}\, dA(\zeta)=\int_{D\lpar 0;\La\sqrt{\rho}\rpar}R_{n}\, dA(z)$$
and
\begin{align*}\Tr T_{n,\La}^{\,2}&=\iint_{\C^2} \babs{\,\bfK_n(\zeta,\eta)\,}^{\,2}\,\1(\zeta)\1(\eta)\, dA(\zeta)dA(\eta)\\
&=\iint_{|\,z\,|,|\,w\,|\le \La\sqrt{\rho}}\babs{K_{n}(z,w)}^{\,2}\, dA(z)dA(w).
\end{align*}
Taking the limit along the given subsequence, we obtain two cases. If $R\equiv 0$ then manifestly $\Tr(T_{n_k,\La}-T_{n_k,\La}^{\,2})\to 0$ as $k\to \infty$. If $R>0$ everywhere, then $R=R^{m}$ for some $m\in \R\cup\{+\infty\}$ where
$R^{m}(z):=F(z+\bar{z}-2m)$. If $m=\infty$, then $R=1$ and
the corresponding limiting kernel is $K=G$.

We will here concentrate on the more subtle case when $m<\infty$, leaving the details for $m=\infty$ to the reader.
Thus, we write
$$\lim_{k\to\infty}\Tr\left(T_{n_k,\La}-T_{n_k,\La}^{\,2}\right)=
\int_{|\,z\,|\le \La\sqrt{\rho}}R^{m}(z)F_\La^{m}(z)\, dA(z)$$
where
$$F_\La^{m}(z):=1-\int_{|\,w\,|\le \La\sqrt{\rho}}B^{m}(z,w)\, dA(w).$$
Here $B^m$ is the Berezin kernel corresponding to $R^m$.

Now fix $\eps$ with $0<\eps<1$. By Lemma \ref{leftlem2} we can choose $M$ large enough that, for all $\La>0$,
$$\frac 1 {\La^2}\int_{\re z>M,\,|z|\le\La\sqrt{\rho}}R^m(z)\, dA(z)<\eps.$$ Since $F_L^m\le 1$, this gives
\begin{equation}\label{mid}\frac1 {\La^2}\lim_{k\to\infty}\Tr\left(T_{n_k,\La}-T_{n_k,\La}^{\,2}\right)<
\frac 1 {\La^2}\int_{|\,z\,|\le \La\sqrt{\rho},\, \re z\le M}R^{m}(z)F_\La^{m}(z)\, dA(z)+\eps.
\end{equation}

Now note that if $|z|\le\La\sqrt{\rho}(1-\eps)$, then
$$F_L^m(z)\ge 1-\int_{|w-z|\le\La\sqrt{\rho}\eps}B^m(z,w)\, dA(w).$$
The right hand side equals $1-\mu(\La\sqrt{\rho}\eps)$ where $\mu$ is as in Lemma \ref{uniint}, so we
can choose $\La_0$ large enough that
$$F_\La^m(z)<\eps,\qquad \text{when}\quad |\,z\,|\le \La\sqrt{\rho}(1-\eps),\quad \re z\le M,\quad \La\ge \La_0.$$
This gives, if $\La\ge\La_0$
\begin{align*}\frac 1 {\La^{\,2}}\int_{|\,z\,|\le \La\sqrt{\rho},\,\re z\le M}R^mF_\La^m& \le
\frac 1 {\La^{\,2}}\left(\int_{\La\sqrt{\rho}(1-\eps)\le|\,z\,|\le \La\sqrt{\rho}}
+\int_{|\,z\,|\le \La\sqrt{\rho}(1-\eps),\, \re z\le M}F_{\La}^m\right)\\
&\le \rho[1-(1-\eps)^2]+\eps\rho(1-\eps)^2,
\end{align*}
where we used that $R^m\le 1$ and $F_\La^m\le 1$.
In view of the relation \eqref{mid},
\begin{equation}\label{vil}\frac 1 {\La^{\,2}} \lim_{k\to\infty}\Tr\left(T_{n_k,\La}-T_{n_k,\La}^{\,2}\right)<
(3\rho+1)\eps,\qquad (\La\ge \La_0).\end{equation}
We have shown that every subsequence $n_k$ has a further subsequence such that \eqref{vil} holds.
This proves the limit \eqref{traces}.
\end{proof}

The next lemma does not presuppose any translation invariance.

\begin{lem} \label{flott} Let
$p\in\Bulk S$ or $p\in\d^* S$. Then
\begin{equation}\label{tr1}\lim_{\La\to\infty}\lim_{n\to\infty}\frac 1 {\La^{\,2}} \Tr T_{n,\La}= \begin{cases} \rho, &p\in\Bulk S,\cr
\rho/2,&p\in\d^* S.
\end{cases}
\end{equation}
\end{lem}

\begin{proof}
Pick a subsequence $n_k$
such that the limit
\begin{equation}\label{costt}\lim_{n_k\to\infty}\sqrt{n_k\Lap Q(p_{n_k})}\dist(p_{n_k},\d S)
\end{equation}
exists, possibly being infinite. Let us write $a=a(n_k)$ for the limit.
Recall that, for any limiting one-point function $R=\lim R_{n_{k_l}}$ we have, by Lemma \ref{leftlem2},
\begin{equation}\label{cholm}\babs{R(z)-\1_{(-\infty,a)}(x)}\le Ce^{-c(x-a)^2},\qquad x=\re z.\end{equation}
If $p\in\bulk S$ then $a=\infty$ we have $R=1$ by Lemma \ref{gininfty}, and
$$\lim_{\La\to\infty}\frac 1 {\La^{\,2}}\lim_{n\to\infty}\Tr T_{n,\La}=
\lim_{\La\to\infty}\frac 1 {\La^{\,2}}\int_{D(0, \La\sqrt{\rho})}dA=\rho.$$

If $p\in\d^* S$ then $a<\infty$, then \eqref{cholm} gives
$$\lim_{\La\to\infty}\frac 1 {\La^{\,2}}\lim_{l\to\infty}\Tr T_{n_{k_l},\La}=
\lim_{\La\to\infty}\frac 1 {\La^{\,2}}\int_{D(0, \La\sqrt{\rho})}\chi_{(-\infty,a)}\left(\re z\right)
dA(z)=\rho/2.$$
The proof of the lemma is complete.
\end{proof}

\subsection{Lower density of $M_\rho$-families.} Let $\config=\{\config_n\}_1^\infty$ be a $2s$-separated $M_\rho$-family contained in $\drop$. Consider a moving point $p=p_n$ in $\drop$.
 We shall prove the following theorem.

\begin{thm} \label{Mclass} If $p\in\bulk S$ then $D^-(\config;p)\ge\rho$. If $p\in \d^*S$ then $D^-(\config;p)\ge \rho/2$.
\end{thm}

To prepare, we note that the identity
$$\left\langle T_{n,\La} f,f\right\rangle=\int_{A_n\left(p,\La\right)}\babs{\,f\,}^{\,2}$$
shows that $T_{n,\La}$ is a (strictly) positive contraction on the space $\calW_{n\rho}$ of weighted polynomials. Let $\lambda^{\,n\rho}_j=\lambda^{\,n\rho}_j(p,\La)$ denote the eigenvalues of $T_{n,\La}$, counted
with multiplicities, arranged in decreasing order;
$$1>\lambda^{\,n\rho}_1\ge \lambda^{\,n\rho}_2\ge\cdots\ge \lambda^{\,n\rho}_{n\rho}>0.$$
Let us write
$$N_n^+(p):=N_n\left(p,\La+s;\config\right)=\#\left\{\config_n\cap A_n(p,\La+s)\right\}.$$

The following lemma, as well as the proof, is borrowed from \cite{AOC}, Lemma 4.1.

\begin{lem} \label{phlegm} Suppose that $\config$ is a $2s$-separated $M_\rho$-family, $\config_n=\{\zeta_{nj}\}_{j=1}^n$. Then there is a constant $\gamma<1$ and a number $n_0$
such that whenever $p\in S$ and $n\ge n_0$, we have $\lambda^{\,n\rho}_j<\gamma$ where $j=N_n^+(p)$.
\end{lem}

\begin{proof} Write $N_n^+=N_n^+(p)$. By the assumed separation we have $N_n^+\le C$ for some constant $C$ independent of $n$.

Let $(\phi_j)_{1}^{n\rho}$ be an orthonormal basis for $\calW_{n\rho}$ consisting of eigenfunctions of $T_{n,\La}$ corresponding to the eigenvalues $\lambda^{\,n\rho}_j$. Fix $p\in S$
and choose constants $c_j$ (not all zero) so that the function $f=\sum_{j=1}^{N_n^+}c_j\phi_j$ satisfies $f(\zeta_{nj})=0$ for all
$\zeta_{nj}\in \config_n\cap A_n^+(p)$. This is possible for all large $n$ (say $n\ge n_0$) since $N_n^+\le C$.

Since $\config$ is $2s$-separated and of class $M_\rho$, we have
\begin{equation}\label{vice}\int_{\drop_s}\babs{\,f\,}^{\,2}\le C\frac 1 n\sum_{\zeta_{nj}\in \drop\setminus A_n(p,\La+s)}\babs{\,f(\zeta_{nj})\,}^{\,2}\le Cs^{-2}\int_{\drop_s\setminus A_n(p,\La)}
\babs{\,f\,}^{\,2},\end{equation}
where we have used Lemma \ref{seppo2} to get the last inequality.

Since $T_{n,\La}$ is the orthogonal projection of $L^2$ on $\Pol_{n\rho}$ we also have
\begin{equation}\label{roy}\sum_{j=1}^{n\rho}\lambda^{\,n\rho}_j\,\babs{\,c_j\,}^{\,2}=\left\langle T_{n,\La}f,f\right\rangle=\int_{A_n(p,\La)}\babs{\,f\,}^{\,2}.\end{equation}
By \eqref{vice} and \eqref{roy}, we find that
\begin{align*}\lambda^{\,n\rho}_{N_n^+}&\sum_{j=1}^{N_n^+}\babs{\,c_j\,}^{\,2}\le \sum_{j=1}^{n\rho}\lambda^{\,n\rho}_j\babs{\,c_j\,}^{\,2}=\left(\int_{\drop_s}-\int_{\drop_s\setminus A_n(p,\La+s)}\right)
\babs{\,f\,}^{\,2}\\
&\le\left(1-\frac {s^2} C\right)\int_{\drop_s}\babs{\,f\,}^{\,2}\le \left(1-\frac {s^2} C\right)\left\|\,f\,\right\|^{\,2}=\left(1-\frac {s^2} C\right)
\sum_{j=1}^{N_n^+}\babs{\,c_j\,}^{\,2}.
\end{align*}
This shows that $\lambda^{\,n\rho}_{N_n^+}\le \gamma$ where we may take $\gamma=1-s^2/C<1$.
\end{proof}

We now prove Theorem \ref{Mclass}. Thus fix a $2s$-separated $M_\rho$-family $\config$ and a suitable point $p=p_n\in \drop$.
As before, we let $\lambda^{\,n\rho}_j=\lambda^{\,n\rho}_j(p,\La)$ denote the eigenvalues of the concentration operator $T_{n,\La}$,
in non-increasing order.

Let $\delta_\lambda$ denote Dirac measure at $\lambda$.
Define a measure $\mu_n$ by $\mu_n=\sum_{j=1}^{n\rho}\delta_{\lambda^{\,n\rho}_j}$. Then
$$\Tr T_{n,\La}=\int_0^1x\,d\mu_n(x)\quad,\quad \Tr T_{n,\La}^{\,2}=\int_0^1 x^2\, d\mu_n(x).$$
Let $\gamma$ and $n_0$ be as in Lemma \ref{phlegm}. Then for $\gamma\ge n_0$,
\begin{align*}\#\left\{j;\, \lambda^{\,n\rho}_j>\gamma\right\}&=\int_\gamma^1 d\mu_n(x)\ge\int_0^1 x\, d\mu_n(x)-\frac 1 {1-\gamma}\int_0^1 x(1-x)\, d\mu_n(x)\\
 &=\Tr T_{n,\La}-\frac 1 {1-\gamma}\Tr\left(T_{n,\La}-T_{n,\La}^{\,2}\right).
\end{align*}
Now write $N_n$ for the number of points in the slightly smaller disk,
$N_n=N_n(p,\La;\config)$. (Cf. \eqref{anz1}.)
Note that the $2s$-separation of $\config$ implies that
$$N_n^+-N_n\le C\La$$ where $C$ is a constant depending on $s$. By Lemma \ref{phlegm} we thus have
$$N_n\ge \#\left\{j;\,\lambda^{\,n\rho}_j\ge\gamma\right\}+O(\La),$$
where the $O$-constant is independent of $n$. This gives that
\begin{align*}\liminf_{n\to\infty}&\frac {\#\left\{\config_n\cap A_n(p,\La)\right\}} {\La^{\,2}}\ge\liminf_{n\to\infty}\frac {\#\left\{j;\,\lambda^{\,n\rho}_j\ge\gamma\right\}+O(\La)}
{\La^{\,2}}\\
&\ge\liminf_{n\to\infty}\left[\frac {\Tr T_{n,\La}} {\La^{\,2}}-\frac 1 {1-\gamma}\frac {\Tr \left(T_{n,\La}-T_{n,\La}^{\,2}\right)} {\La^{\,2}}\right]+O(1/\La).
\end{align*}
By the trace estimates in Lemma \ref{flott}, we conclude that $D^-(\config,p)\ge \rho$ if $p\in \Bulk S$, whereas $D^-(\config,p)\ge \rho/2$ if $p\in\d^* S$. $\qed$

\subsection{Upper density of interpolating families} Now suppose that $\config$ is a $\rho$-interpolating family contained in $\drop$. Also fix a moving point $p=p_n$ in $\drop$.

\begin{thm}\label{inter} If $p$ is in the bulk, then $D^+(\config;p)\le \rho$. If $p$ is in the regular boundary regime, then $D^+(\config;p)\le \rho/2$.
\end{thm}

Recall that an interpolating family is uniformly
separated (Theorem \ref{seppo3}). Let $2s$ be a separation constant of $\config$, where $s<\La$.
(Here $\La$ is a fixed parameter with $\La\ge 1$.) We will denote the number of points of $\config_n$ in the disk $A_n(p,\La-s)$ by
$$N_n^-=N_n^-(p):=N_n(p,\La-s,\config).$$
It is easy to see that $N_n^-\le M$ for some constant $M$ depending only on $s$.

We now order the points in $\config_n=\{\zeta_{n1},\ldots,\zeta_{nn}\}$ so that the first $N_n^-$ elements, i.e., the points
 $\zeta_{n1},\ldots,\zeta_{nN_n^-}$, are the elements of $\config_n\cap A_n(p,\La-s)$. Since $\config$ is interpolating we can find functions $f_{nj}\in\Pol_{n\rho}$
such that $f_{nj}(\zeta_{jk})=\delta_{jk}$ and $\left\|\,f_{nj}\,\right\|^{\,2}\le C/n$. By repeated use of the Cauchy--Schwarz inequality, the function $f=\sum_{j=1}^{N_n^-}c_jf_{nj}$
then satisfies
\begin{align*}\left\|\,f\,\right\|^{\,2}&\le\sum_{j,k=1}^{N_n^-}|\,c_j\,||\,c_k\,|\cdot
\babs{\left\langle f_{nj},f_{nk}\right\rangle}\le \frac C n\sum_{j,k=1}^{N_n^-}|\,c_j\,||\,c_k\,|\\
&\le C\frac 1 n\sum_{j=1}^{N_n^-}|\,c_j\,|\sqrt{N_n^-}\sqrt{\sum_{k=1}^{N_n^-}|\,c_k\,|^{\,2}}\le CN_n^-\frac 1 n \sum_{j=1}^{N_n^-}|\,c_j\,|^{\,2}\le CM\frac 1 n \sum_{j=1}^{N_n^-}|\,c_j\,|^{\,2}.
\end{align*}
Since $f(\zeta_{nj})=c_{nj}$, Lemma \ref{seppo2} implies the estimate
\begin{equation}\label{impe}\left\|\,f\,\right\|^{\,2}\le C'\int_{A_n(p,\La)}\babs{\,f\,}^{\,2}=C'\left\langle T_{n,\La}f,f\right\rangle
\end{equation}
where $T_{n,\La}$ is concentration operator, $C'$ is a constant independent of $n$.

The function $f=\sum c_{nj}f_{nj}$ used above belongs to the following linear
span
$$F_{N_n^-}:=\operatorname{span}\left\{f_{nj};\, j=1,\ldots,N_n^-\right\}\subset\Pol_{n\rho}.$$
By \eqref{impe} we have, with $\delta=1/C'>0$,
\begin{equation}\label{fischer}\left\langle T_{n,\La}f,f\right\rangle\ge \delta\left\|\,f\,\right\|^{\,2},\qquad f\in F_{N_n^-}.\end{equation}

As before, denote by $\lambda^{\,n\rho}_j$ the eigenvalues of the operator $T_{n,\La}$ on $\Pol_{n\rho}$, ordered in decreasing order.
By Fischer's principle (see \cite{L}, p. 319) we have the estimate
$$\lambda^{\,n\rho}_j=\max_{U_j}\min_{f\in U_j}\frac {\left\langle T_{n,\La} f,f\right\rangle}
{\left\langle f,f\right\rangle},$$
where $U_j$ ranges over the $j$-dimensional subspaces of $\Pol_{n\rho}$. Since the space $F_{N_n^-}$ has dimension $N_n^-$, the estimate \eqref{fischer}
shows that $\lambda^{\,n\rho}{N_n^-}\ge\delta$. Hence
\begin{equation}\label{prece}N_n^-\le\#\left\{j;\, \lambda^{\,n\rho}_j\ge\delta\right\}.\end{equation}
By the separation of $\config$ we have the estimate
$$N_n-N_n^-\le C\La.$$ The inequality \eqref{prece} therefore implies
\begin{equation}\label{price}N_n\le\#\left\{j;\, \lambda^{\,n\rho}_j\ge\delta\right\}+O(\La).\end{equation}
Consider again the measure $\mu_n=\sum_{j=1}^n\delta_{\lambda^{\,n\rho}_j}$; we this time use the estimate
$$\#\left\{j;\,\lambda^{\,n\rho}_j\ge\delta\right\}=\int_\delta^1 d\mu_n(x)\le\int_0^1x\, d\mu_n(x)+\frac 1 \delta\int_0^1 x(1-x)\, d\mu_n(x)$$
which is valid for $0<\delta<1$. Applying this estimate to \eqref{price}, we find that
\begin{align*}\#\left\{\config_n\cap A_n(p,\La)\right\}&=N_n\le \Tr T_{n,\La}+\frac 1 \delta\left[\Tr\left(T_{n,\La}-T_{n,\La}^{\,2}\right)\right]+O(\La).\end{align*}
The trace estimates in Lemma \ref{flott} now give that
$$\limsup_{n\to\infty}\frac {\#\left\{\config_n\cap A_n(p,\La)\right\}}
{\La^{\,2}}\le \begin{cases} \rho+O(1/\La) & p\in\bulk S\cr
\rho/2+O(1/\La)&p\in\d^*S\cr
\end{cases}.$$
Taking the $\limsup$ as $\La\to\infty$, we see that
$D^+(\config;p)\le \rho$ if $p$ is in the bulk while $D^+(\config;p)\le \rho/2$ if $p$ is in the regular boundary regime.
The proof of Theorem \ref{inter} is complete. $\qed$

\section{The density theorem}

In this section, we consider a family $\calF=\{\calF_n\}$ of Fekete sets.
We shall prove the density theorem (theorems \ref{THAA}-\ref{THB}).
The proofs of theorems \ref{THAAA} and \ref{THA} are slightly more transparent
when the boundary $\d S$ is everywhere regular. For that reason, we will first give an argument for the regular case, and then give the modifications needed in the presence of singular boundary points.

\subsection{Proof of Theorem \ref{THAA}} Let $\calF_n=\{\zeta_{n1},\ldots,\zeta_{nn}\}$ be an $n$-Fekete set. We define
weighted Lagrangian interpolation polynomials $\ell_{nj}\in\Pol_n$ by
$$\ell_{nj}(\zeta)=\lpar\prod_{i\ne j}(\zeta-\zeta_{ni})/\prod_{i\ne j}(\zeta_{nj}-\zeta_{ni})\rpar\cdot e^{-n(Q(\zeta)-Q(\zeta_{nj}))/2}.$$
Observe that $\ell_{nj}(\zeta_{nk})=\delta_{jk}$
and
$$\babs{\,\ell_{nj}(\zeta)\,}^{\,2}\le 1,\qquad \zeta\in \C.$$
Hence, by Lemma \ref{bern}
\begin{equation}\label{stopp}\babs{\,\nabla\lpar\babs{\,\ell_{nj}\,}\rpar(\zeta)\,}\le \sqrt{e\,n\,\Lap Q(\zeta)}(1+o(1)),\qquad \zeta\in \nbh\setminus \calF_n,\end{equation}
where $\nbh$ is a suitable neighbourhood of $S$. (It is easy to see that the proof of Lemma \ref{bern} goes through
also when $\zeta\in\nbh\setminus S$.)

Fix a point $\zeta_{nj}\in\calF_n$ and assume that another point $\zeta_{nk}$ is sufficiently close to $\zeta_{nj}$, say $\babs{\zeta_{nj}-\zeta_{nk}}\le C/\sqrt{n}$ where $C$ is some large constant. (If there are no such $\zeta_{nk}$, there is nothing to prove.)

Integrating \eqref{stopp}
with respect to arclength over the line-segment $\gamma$ joining $\zeta_{nj}$ to $\zeta_{nk}$ (or perhaps a slightly perturbed curve to avoid other points $\zeta_{nj}$ on that segment) we find that
\begin{equation}\label{sto2}\begin{split}1=\babs{\,\babs{\,\ell_{nj}(\zeta_{nj})\,}-\babs{\,\ell_{nj}(\zeta_{nk})\,}\,}&\le (1+o(1))\int_\gamma
\sqrt{e n\Lap Q(\zeta)}\, |d\zeta|\\
&\le\sqrt{en\Lap Q(\zeta_{nj})}(1+o(1))\babs{\,\zeta_{nj}-\zeta_{nk}\,},\\
\end{split}
\end{equation}
where we used continuity of $\Lap Q$ to replace the supremum over $\gamma$ by $\Lap Q(\zeta_{nj})$.

The estimate \eqref{sto2} means that
$$d_{n}(\zeta_{nj}):=\sqrt{n\Lap Q(\zeta_{nj})}\min_{k\ne j}\babs{\,\zeta_{nk}-\zeta_{nj}\,}$$
satisfies $d_n(\zeta_{nj})\ge (1+o(1))/\sqrt{e}$ for all $j$. Hence
$$\Delta(\calF)=\liminf_{n\to\infty}\min_{j=1,\ldots,n}d_n(\zeta_{nj})\ge 1/\sqrt{e}.$$
The proof of Theorem \ref{THAA} is complete. q.e.d.

\subsection{Fekete sets and interpolating families; the regular case}\label{irc}

\begin{thm} \label{ipo} Let $\calF=\{\calF_n\}$ be a family of Fekete sets. If the boundary of the droplet is everywhere regular, then $\calF$ is $\rho$-interpolating
for each $\rho>1$.
\end{thm}

\begin{proof}
For a fixed $\eps>0$ we define weighted polynomials $L_{nj}\in\Pol_{(1+2\eps)n}$ by
\begin{equation}\label{ulf}L_{nj}(\zeta)=\lpar \frac {\mathbf{K}_{\eps n}(\zeta,\zeta_{nj})}
{\mathbf{R}_{\eps n}(\zeta_{nj})}\rpar^2\cdot \ell_{nj}(\zeta).\end{equation}
Observe that $L_{nj}(\zeta_{jk})=\delta_{jk}$.

Since $\d S$ is everywhere regular, we can use
Lemma \ref{beloww} and the inclusion $\calF_n\subset \drop$ to assert the existence of $c>0$ such that
\begin{equation}\label{ue}\bfR_{\eps n}(\zeta_{nj})\ge c\eps n.\end{equation}
Hence
\begin{equation}\label{ulv}
\sum_{j=1}^\infty \babs{\,L_{nj}(\zeta)\,}\le \frac C {n^2}\sum_{j=1}^n \babs{\,\mathbf{K}_{\eps n}(\zeta,\zeta_{nj})\,}^{\,2}\le
\frac C {n^2}\int \babs{\,\mathbf{K}_{\eps n}(\zeta,\eta)\,}^{\,2}~dA(\eta)\le \frac {C'} n,\end{equation}
where we have used \eqref{ulf} together with Theorem \ref{THAA}
and Lemma \ref{seppo2}.

Note that, since $\babs{\,\ell_{nj}\,}\le 1$, the estimate \eqref{ue} and the reproducing property gives
\begin{equation}\label{upp}\norm{L_j}_{L^1}\le \frac 1 {c\eps n}\int_\C \frac {\babs{\,\bfK_{\eps n}(\zeta,\zeta_{nj})\,}^{\,2}} {\bfR_{\eps n}(\zeta_{nj})}\, dA(\zeta)=\frac C n\end{equation}
where $C=1/c\eps$.

Now define $T:\C^n\to L^1+L^\infty$ by $T(c)=\sum_{j=1}^n c_jL_{nj}$.
By \eqref{upp} we have
$$\norm{\,T\,}_{\ell^1_n\to L^1}\le \sup_j \left\|\,L_{nj}\,\right\|_{L^1}\le \frac C n.$$
Moreover, with $F_n(\zeta)=\sum\babs{\,L_{nj}(\zeta)\,}$, we have by \eqref{ulv}
$$\norm{\,T\,}_{\ell^\infty_n\to L^\infty}\le \left\|\,F_n\,\right\|_{L^\infty}\le C.$$
By the Riesz-Thorin theorem,
$$\norm{\,T\,}_{\ell^2_n\to L^2}\le \frac C {\sqrt{n}}.$$
Hence if $f=T(c)$ then $f\in \Pol_{n(1+2\eps)}$, $f(\zeta_{nj})=c_j$ and
$$\int\babs{\,f\,}^{\,2}\le C\frac 1 n \sum_{j=1}^n \babs{\,f(\zeta_{nj})\,}^{\,2}.$$
We have shown that $\mathcal{F}$ is $(1+2\eps)$-interpolating.
\end{proof}

\subsection{Fekete sets and $M$-families; the regular case}\label{mrc}

\begin{thm} \label{mro} Assume that the boundary of $S$ be everywhere regular and
let $\calF=\{\calF_n\}$ be a family of Fekete sets. Then $\calF$ is an $M_\rho$-family for each $\rho<1$.
\end{thm}

\begin{proof}
Fix a function $f\in \Pol_{n(1-2\eps)}$ and put
$$g_\zeta(\eta)=f(\eta)\cdot \lpar \frac {\mathbf{K}_{n\eps}(\eta,\zeta)}
{\mathbf{R}_{n\eps}(\zeta)}\rpar^2.$$
By Lagrange interpolation,
$$g_\zeta(\eta)=\sum_{j=1}^n g_\zeta(\zeta_{nj})\ell_{nj}(\eta).$$
Hence if we put
$$\tilde{L}_{nj}(\zeta)=\lpar \frac {\mathbf{K}_{n\eps}(\zeta_{nj},\zeta)}
{\mathbf{R}_{n\eps}(\zeta)}\rpar^2\cdot \ell_{nj}(\zeta),$$
we will have
$$f(\zeta)=g_\zeta(\zeta)=\sum_{j=1}^n f(\zeta_{nj})\tilde{L}_{nj}(\zeta).$$
By Lemma \ref{beloww} and the assumed regularity of $\d S$ we have the uniform estimate
\begin{equation}\label{kos}\bfR_{\eps n}(\zeta)\ge c\eps n,\qquad \zeta\in \drop_s,\end{equation}
where the positive constant $c$ depends on $s$. Here $S_s=S+D(0,s/\sqrt{n})$.

Since $\babs{\,\ell_{nj}\,}\le 1$, the reproducing property and \eqref{kos} gives
\begin{equation}\label{los}\left\|\,\tilde{L}_{nj}\,\right\|_{L^1(\drop_s)}\le \frac C n\end{equation}
where $C$ depends on $\eps$ and $s$.

Consider the function $\tilde{F}_n(\zeta)=\sum_{j=1}^n \babs{\,\tilde{L}_{nj}(\zeta)\,}$. Using the estimate \eqref{kos},
we have the estimate
$$\babs{\,\tilde{L}_{nj}(\zeta)\,}\le C\frac 1 {n^2}\babs{\,\bfK_{n\eps}(\zeta,\zeta_{nj})\,}^{\,2},\qquad z\in \drop_s$$
where $C$ depends on $\eps$ and $s$. Using this estimate, the argument leading to the estimate \eqref{ulv} shows that
\begin{equation}\label{opla}\norm{\,\tilde{F}_n\,}_{L^\infty(\drop_s)}\le C.\end{equation}

Now define $\tilde{T}:\C^n\to \lpar L^1+L^\infty\rpar\lpar \drop_s\rpar$ by $\tilde{T}(c)=\sum c_j\tilde{L}_j$.
Then \eqref{los} and \eqref{opla} show that $\norm{\,\tilde{T}\,}_{\ell^1_n\to L^1(\drop_s)}\le C/n$ and $\norm{\,\tilde{T}\,}_{\ell^\infty_n\to L^\infty(\drop_s)}\le C$.
So by interpolation we have
$$\norm{\,\tilde{T}\,}_{\ell^2_n\to L^2(\drop_s)}\le \frac C {\sqrt{n}}.$$

We have shown that an arbitrary $f\in\Pol_{n(1-2\eps)}$ has the representation $f=\tilde{T}(c)$ where $c_j=f(z_j)$. It follows that
$$\int_{S_s}\babs{\,f\,}^{\,2}\le C\frac 1 n \sum_{j=1}^n\babs{\,f(z_{nj})\,}^{\,2},\quad f\in \Pol_{n(1-2\eps)}.$$
By definition, this means that $\mathcal{F}\in M_{1-2\eps}$. 
\end{proof}

\subsection{Proof of theorems \ref{THAAA} and \ref{THA} for regular droplets} \label{mpf} Let $p=p_n$ be a moving point in $S$, and let $\calF=\{\calF_n\}$ be a family
of Fekete sets. We here assume that the boundary of $S$ is everywhere regular.

Fix $\rho_1>1$. By Theorem \ref{ipo} $\calF$ is then $\rho_1$-interpolating. We can then apply Theorem \ref{inter} with $\config=\calF$.
The result is that $D^+(\calF;p)\le \rho_1$ if $p\in \bulk S$ and $D^+(\calF;p)\le \rho_1/2$ if $p\in\d^* S$.

On the other hand, if $\rho_2<1$, then Theorem \ref{mro} shows that $\calF$ is a $M_{\rho_2}$-family. Hence Theorem \ref{Mclass} may be applied. We conclude that $D^-(\calF;p)\ge \rho_2$ if $p\in \Bulk S$ while $D^-(\calF;p)\ge \rho_2/2$ if $p\in\d^* S$.

We have shown that $\rho_2\le D^-(\calF;p)\le D^+(\calF;p)\le\rho_1$ if $p\in \bulk S$ and
$\rho_2/2\le D^-(\calF;p)\le D^+(\calF;p)\le\rho_1/2$ if $p\in\d ^* S$. Letting $\rho_2\uparrow 1$ and $\rho_1\downarrow 1$ finishes the proof. q.e.d.

\subsection{Proof of Theorem \ref{THB}} Let
 $q=q_n\in S$ denotes a moving point in the singular boundary regime $\d'S$. There is then
 a singular boundary point $p_*$ of type $\nu$ and a constant $M$ such that $\dist(q_n,p_*)\le
Mn^{-1/\nu}$. We claim that, for all large $n$, $q_n$ has distance at most
$T/\sqrt{n\Lap Q(p_*)}$ from the boundary, where $T$ is a constant depending on $M$.

To verify this, we consider the case when $p_*=0$ is a cusp of type $\nu$; then in a
suitable coordinate system, the boundary
of $S$ near $0$ looks roughly like $y^2=cx^\nu$ with a suitable constant $c>0$. (Cf. Section \ref{drp}.) Let $q_n=x_n+iy_n$ be a point in $S$ with $|q_n|\le Mn^{-1/\nu}$.
Then
$0\le x_n\le Mn^{-1/\nu}$ and $\dist(x_n,\d S)\le\const\cdot cx_n^\nu$. Hence
for all large $n$ we have the estimates
\begin{align*}\dist(q_n,\d S)&\le C\dist(x_n,\d S)\le C'x_n^{\nu/2}\\
&=C'M^{\nu/2}\cdot n^{-1/2}= T/\sqrt{n\Lap Q(p_*)},
\end{align*}
where $T=C'M^{\nu/2}\sqrt{\Lap Q(p_*)}$. The case when $p_*$ is a double point is similar.

Take a number $\La>T$ and
consider the disk
$$\tilde{A}_n(q,\La):=D\left(q_n,\frac \La {\sqrt{n\Lap Q(p_*)}}\right).$$
The image of $S\cap \tilde{A}_n(q,\La)$ under the appropriate rescaling
\begin{equation}\label{ner}z=e^{-i\theta_n}\sqrt{n\Lap Q(p_*)}(\zeta-q_n)\end{equation} is then contained in the
truncated strip
$$U:=\{x+iy;\, -T\le x\le T,\,-\La\le y\le \La\}.$$
Referring to the mapping \eqref{ner} we write
$$\tilde{\calF}_{n,q,\La}=\left\{z_{nj};\,\zeta_{nj}\in \tilde{A}_n(q,\La)\right\}$$
for the image of the Fekete points inside $\tilde{A}_n(q,\La)$.

By uniform separation of Fekete sets (see Theorem \ref{THAA})
 there is $c>0$ such that any two distinct points $z_{nj},z_{nk}\in
\tilde{\calF}_{n,q,\La}$ have distance $\babs{z_{nj}-z_{nk}}\ge c$.
This implies that there is a constant $C$ such that the number $N_n(q,\La,\calF):=\#\tilde{\calF}_{n,q,\La}$ satisfies
$$N_n(q,\La,\calF)\le C\operatorname{meas}(U)\le C'T\La.$$
It follows that the upper Beurling-Landau density satisfies
$$D^+(\calF,q)=\limsup_{\La\to\infty}\limsup_{n\to\infty}\frac {N_n(q,\La,\calF)}
{\La^{\,2}}\le \limsup_{\La\to\infty}\frac {CT}{\La}=0.$$
We have shown that $D(\calF,q)=0$, as desired. q.e.d.

\subsection{Proof of theorems \ref{THAAA} and \ref{THA} with singular boundary points}
In the presence of singular boundary points we will "cut away'' the singular boundary regime from the droplet and consider a slightly smaller set $S_n$, which is still large enough to contain most Fekete points. Here is the construction.

Fix a sequence $C_n$ of positive numbers with $C_n\to\infty$ as $n\to\infty$ ("slowly'').

We define "regularized droplets'' by
$$S_n:=S\setminus\bigcup D\left(p_*,C_nn^{-1/\nu_*}\right),$$
where the union extends over all singular boundary points $p_*$ of types $\nu_*$.

Now pick a family $\calF=\{\calF_n\}$ of Fekete sets and put
$$\calF_n':=\calF_n\cap S_{n},\quad m_n:=\#\calF_n'.$$
We order the points so that $\calF_n'=\{\zeta_{n1},\ldots,\zeta_{nm_n}\}$.

\begin{defn} Let $\config=\{\config_n\}$, $\config_n=\{\zeta_{nj}\}_{j=1}^{m_n}$ be any family of points of $S_n$ and $\rho$ a positive number (close to $1$).
\begin{enumerate}[label=(\roman*)]
\item $\config$ is said to be \textit{$\rho$-interpolating} if for all complex sequences $\{c_{nj}\}_{j=1}^{m_n}$ there exists a weighted polynomial $f\in\Pol_{n\rho}$ such that $f(\zeta_{nj})=c_{nj}$ for all $j$ and
    $\left\|\, f\,\right\|^{\,2}\le C\frac 1 n\sum_{j=1}^{m_n}\babs{\,c_{nj}\,}^{\,2}$.
\item We say that $\config$ is an \textit{$M_{\rho,C_n}$-family} if it is uniformly $2s$-separated for some $s>0$ and, with
$S_{n,s}=S_n+D(0,s/\sqrt{n})$,
$$\int_{S_{n,s}}\babs{\, f\,}^{\,2}\le \const \frac {1} n\sum_{j=1}^{m_n}\babs{\, f(\zeta_{nj})\,}^{\,2}.$$
\end{enumerate}
\end{defn}

The proofs of theorems \ref{THAAA} and \ref{THA} can now be finished as in the regular case, by noting that
theorems \ref{inter} and \ref{Mclass} have the following counterparts for the above types of families.

\begin{lem} \label{lalin2} Let $\config=\{\config_n\}$, $\config_n=\{\zeta_{nj}\}_{j=1}^{m_n}$ be a family contained in $S_n$. Then
\begin{enumerate}[label=(\roman*)]
\item If $\config$ is $\rho$-interpolating and $p$ is any moving point in $S$ then $D^+(\config,p)\le \rho$ if
$p\in\bulk S$ and $D^+(\config,p)\le \rho/2$ if $p\in\d^* S$.
\item \label{ll22} If $\config$ is of class $M_{\rho,C_n}$ and $p_n\in S_{n/2}$ for all $n$ then $D^-(\config,p)\ge \rho$ if
$p\in\bulk S$ and $D^-(\config,p)\ge \rho/2$ if $p\in\d^*S$.
\end{enumerate}
\end{lem}

\begin{proof}[Remark on the proof] The estimates in Section \ref{conop} persist to hold in the present more general situation, with "$S$'' replaced by "$S_n$''. We omit repeating the simple details here.
\end{proof}

We now invoke Lemma \ref{beloww}, which implies that for all moving point $p=p_n\in S_{n,s}$.
In particular,
$\bfR_n(p_n)\ge cn$ where $c$ is a positive constant, whence (for small $\eps>0$)
\begin{equation}\label{body}\bfR_{\eps n}(\zeta_{nj})\ge c\eps n,\qquad (\zeta_{nj}\in\calF_n').\end{equation}
These estimates contain what we need in order to generalize the arguments from the regular cases above.

\begin{lem} \label{lalin} The family $\calF'$ is $\rho$-interpolating for each $\rho>1$ and of class $M_{\rho,C_n}$ when
$\rho<1$.
\end{lem}

\begin{proof}[Remark on the proof]
The proof is as before (sections \ref{irc} and \ref{mrc}) by using $\calF_n'$ instead of $\calF_n$ and applying the estimate \eqref{body}.
\end{proof}

We now finish the proofs of theorems \ref{THAA} and \ref{THA}.
 If $p=p_n$ is a moving point in $S_{n/2}$ (to insure that Lemma \ref{lalin2} \ref{ll22} can be applied)
then the families $\calF$ and $\calF'$ have the same upper and lower densities at $p$, i.e.
$D^+(\calF,p)=D^+(\calF',p)$ and $D^-(\calF,p)=D^-(\calF',p)$. Moreover $D^+(\calF',p)$ is at least $1/2$ if $p\in\d^* S$ and at least $1$ if $p\in\Bulk S$ by lemmas \ref{lalin2} and \ref{lalin}. By the same token,
$D^-(\calF',p)$ is at least $1$ if $p\in\Bulk S$ and at least $1/2$ if $p\in \d^*S$. q.e.d.

\section{Concluding remarks: real-analyticity and boundary regularity} \label{concrem}
In the foregoing, we have frequently made use of analytic properties of the potential and of the
boundary of a droplet. We here give some additional remarks on the nature of these properties. We shall also give a short proof that the relevant properties of the boundary
follow from our standing assumptions of the potential, via Sakai's theorem from \cite{Sa}. This was noted in \cite{Sa}, cf. \cite{HS,LM} for other discussions.

 The next section gives a brief overview of the role which real-analyticity properties play in our analysis; after that, we will turn to the details behind Sakai's regularity theorem.

\subsection{The role of real-analyticity} \label{trra}
Recall our standing assumption that $Q$ be real-analytic, except that it is possible $+\infty$ in portions of the plane (and everywhere l.s.c.). We now give some additional comments on the nature of this assumption.

Amongst our main results, Theorem \ref{THAA} holds under the more general assumption that $Q$ be $C^2$-smooth, but our proofs of the other main results all use real-analyticity in one way or the other.

For a given potential, the determination of the droplet is a hard problem:
the inverse problem of potential theory. It is therefore preferable to seek results which
do not rely on the detailed knowledge of an individual droplet, but rather on general properties which can be read off directly from the potential. To our knowledge, the only manageable condition which implies smoothness of the boundary uses Sakai's theorem, and thus ultimately some assumption involving real-analyticity.

On the other hand, the class of potentials $Q$ satisfying our standing assumptions (or perhaps slight, $n$-dependent perturbations thereof) are natural for many applications, e.g. in conformal field theory, cf. \cite{KM}. In general, $n$-dependent potentials of the form $Q_n=Q+h/n$, where $h$ is a fixed, possibly non-analytic function, are also quite useful, e.g. for studying statistical properties of Coulomb gases, see \cite{AM,BBNY}. Sakai's regularity theorem can be applied to this situation, since $Q_n$ and $Q$ have
asymptotically the same droplets.

In spite of questions concerning practical applicability, it might seem conceivable that the density results in theorems \ref{THAAA}, \ref{THA} should hold, say for a $C^2$-smooth potential such that the boundary be everywhere $C^1$-regular. We do not wish to speculate too much about that, but remark that our present methods use real-analyticity
already when we define limiting kernels at a moving point: the normal-families argument from \cite{AKM} uses
this assumption. Moreover, the proofs of several other relevant results from \cite{AKM} use real-analyticity in one way or the other.

\subsection{The obstacle function} We give a few definitions, to prepare for our proof of boundary regularity.

Given a Borel measure $\mu$ on $\C$, we denote its logarithmic potential by
$$U^\mu(\zeta)=\int_\C\log\frac 1 {\left|\,\zeta-\eta\,\right|}\, d\mu(\eta).$$
The weighted energy $I_Q[\mu]$ of \eqref{egy} can be written
$I_Q[\mu]=\int U_\mu\, d\mu+\int Q\, d\mu$.

Define the "Robin constant'' $\gamma$ in external potential $Q$ as the minimum of $Q+2U^\sigma$ over $\C$.
By the \textit{obstacle function}, we mean the
subharmonic function
$$\eqpot(\zeta)=-2U^\sigma(\zeta)+\gamma.$$
It is known that $\eqpot=Q$ on $S$ while $\eqpot$ is harmonic on $S^c$ and
has a Lipschitz continuous gradient on $\C$. (See \cite{ST}.) In particular, the droplet
$S$ is contained in the coincidence set $\{Q=\eqpot\}$ and the equilibrium measure can be written
$$d\sigma=\Lap Q\1_S\, dA=\Lap\eqpot\, dA.$$

\subsection{Sakai's regularity theorem} Consider the complement $S^c$ of the droplet. We shall show that locally, close to a point $p_*\in \d S$, the set $S^c$ has a
\textit{Schwarz function} $s$. This means that there
exists a neighbourhood $N$ of $p_*$ an analytic function $s$ on $N\cap S^c$ which extends continuously
to $N\cap \d S$ and satisfies $s(z)=\bar{z}$ there.

Let
$\fii$ be a conformal map of the disk $D(0,1)$ onto a component
$U$ of $S^c$.
Given the (local) existence of a Schwarz function as above, it follows from Sakai's theorem in \cite{Sa} that $\fii$ extends analytically across the boundary of $D$. (Cf. \cite{S,S1}.)

In order to deduce the desired analytical properties of the boundary (see Section \ref{ptd}) it thus remains to show that there exists a local Schwarz function.

To this end, we can
assume that
$p_*=0\in \d S$.
Recalling that we have assumed that $Q$ be real-analytic and strictly subharmonic in a neighbourhood of $S$, we
let $D$ denote a small enough open disk centered at the origin and write
$$Q(z)=\sum_{j,k=0}^\infty a_{jk}z^j\bar{z}^k\quad ,\quad A(z,w)=\sum a_{jk}z^jw^k,\quad (z,w\in D).$$

 In $D\setminus S$, we have
 that $\d\eqpot$ coincides with a holomorphic function whose boundary values on $\d S$ coincide
 with those of $\d Q$. Define a Lipschitz function on $D\times D$ by
 $$G(z,w)=\d_1 A(z,w)-\d \eqpot(z).$$
Notice that $G(0,0)=0$ and $\d_2 G(0,0)=\Delta Q(0)>0$.
Hence, by a version of the implicit function theorem (e.g. \cite{DR}, Section 1E) there is a unique Lipschitzian function
$s(z)$ defined on $D$ such that $G(z,s(z))=0$. This implies that $\bar{\d}s(z)=0$ a.e.
on $D\setminus S$, so $s$ is
holomorphic there. It is also clear that $s$ is continuous up to $(\d S)\cap D$ and that
$s(z)=\bar{z}$ for $z\in \d S$, so $s$ is our required Schwarz function. $\qed$

\end{document}